\documentclass{article}%
\usepackage{amsmath}
\usepackage{amsfonts}
\usepackage{amssymb}
\usepackage{graphicx}%
\providecommand{\U}[1]{\protect\rule{.1in}{.1in}}
\newtheorem{theorem}{Theorem}

\newtheorem{corollary}[theorem]{Corollary}

\newtheorem{proposition}[theorem]{Proposition}
\newtheorem{remark}[theorem]{Remark}

\newenvironment{proof}[1][Proof]{\noindent\textbf{#1.} }{\ \rule{0.5em}{0.5em}}

\begin{document}
\thispagestyle{empty}
\title{\bf The $\varepsilon$-entropy of some infinite dimensional compact ellipsoids and  fractal dimension of attractors}\maketitle

\author{ \center  Mar\'ia ANGUIANO \\ Departamento de An\'alisis Matem\'atico. Facultad de Matem\'aticas. \\Universidad de Sevilla. \\ P. O. Box 1160, 41080-Sevilla (Spain).
\\ anguiano@us.es\\}
\medskip\author{ \center  Alain HARAUX (1, 2)\\ 1.  UPMC Univ Paris 06, UMR 7598, Laboratoire Jacques-Louis Lions, \\ F-75005,
Paris, France.\\{2- CNRS, UMR 7598, Laboratoire Jacques-Louis Lions, \\Bo\^{\i}te courrier 187,  75252 Paris Cedex 05,  France.\\
 haraux@ann.jussieu.fr\\}}

\vskip20pt

 \renewcommand{\abstractname} {\bf Abstract}
\begin{abstract} 
We prove an estimation of the Kolmogorov $\varepsilon$-entropy in $H$ of the unitary ball in the space $V$, where $H$ is a Hilbert space and $V$ is a Sobolev-like subspace of $H$. Then,  by means of Zelik's result \cite{Zelik_2000}, an estimate of the fractal dimension of the attractors of some nonlinear parabolic equations is established.
\end{abstract}
\bigskip\noindent

 {\small \bf AMS classification numbers:} 37L30, 35B41

  \newpage
\section {Introduction}
Let $M$ be a precompact set in a metric space $X$. We recall the definition of the fractal dimension of $M$ (see, for instance, Temam \cite{Temam}). According to Hausdorff criteria the set $M$ can be covered by a finite number of $\varepsilon$-balls in $X$ for every $\varepsilon>0$. Denote by $N_{\varepsilon}(M,X)$ the minimal number of $\varepsilon$-balls in $X$ which cover $M$. Then the Kolmogorov $\varepsilon$-entropy of the set $M$ in $X$ is defined to be the following number
$$\mathcal{H}_{\varepsilon}(M,X)\equiv \log_2 N_{\varepsilon}(M,X),$$
and the fractal dimension of $M$ can be defined in the following way
\begin{eqnarray*}
dim_{F}(M)=dim_{F}(M,X)=\limsup_{\varepsilon \rightarrow 0^+}\frac{\mathcal{H}_{\varepsilon}(M,X)}{\log_2 \frac{1}{\varepsilon}}.
\end{eqnarray*}

In the present paper, we shall be dealing with estimates of the fractal dimension of the invariant sets (attractors) of the semigroups generated by infinite-dimensional dynamical systems. The usual way of estimating the fractal dimension of invariant sets involving the Liapunov exponents and $k$-contraction maps (see, for instance, Temam \cite{Temam}) requires the semigroup to be quasidifferentiable with respect to the initial data on the attractor. It is well known that the Hausdorff dimension is less than or equal to the fractal dimension. In this sense, in \cite{Chepy_Ilyin}, Chepyzhov and Ilyin show that the Hausdorff and fractal dimension have the same upper bound  generalizing to the infinite-dimensional case the method of Chen \cite{Chen}.

To avoid the  differentiability hypothesis,  Zelik, in \cite{Zelik_2000}, presents a new approach to estimate the dimension of invariant sets. The basic tool of his method is the following very general property.
\begin{theorem}[Zelik] \label{Zelik_2000}
Let $V$ and $H$ be Banach spaces, $V$ be compactly embedded in $H$ and let $K$ be a compact subset of $H$. Assume that there exists a map $L:K\rightarrow K$ such that $L(K)=K$ and the following `smoothing' property is valid
\begin{eqnarray}\label{C}
\left\Vert L(k_1)-L(k_2) \right\Vert_{V} \leq C \left\vert k_1-k_2 \right\vert_{H}, \text{   } \forall k_1,k_2 \in K, \text{ }C>0.
\end{eqnarray}
Then, the fractal dimension of $K$ in $H$ is finite and can be estimated in the following way:
\begin{eqnarray*}
dim_{F} (K,H) \leq \mathcal{H}_{\frac{1}{4C}} \left(B_{V}(0,1),H \right),
\end{eqnarray*}
where $C$ is the same as in (\ref{C}) and $B_{V}(0,1)$ means the unit ball centered at $0$ in the space $V$.
\end{theorem}

In the present work, we show (see Theorem \ref{theorem_entropy}) an estimation of the Kolmogorov $\varepsilon$-entropy of $B_{V}(0,1)$ in $H$ where $H$ is a Hilbert space and $V$ is a Sobolev-like subspace of $H$. Then we deduce from  Zelik's result an estimate of the fractal dimension of the  attractor  of some nonlinear parabolic equations in terms of the physical parameters. This result is quite explicit and rather close from the estimate obtained in  \cite{Chepy_Ilyin} under slightly different but quite related assumptions.

\section{Main results}
Let $H$ be a separable Hilbert space with scalar product $(\cdot,\cdot)_H$ and norm $\left\vert \cdot \right\vert_H$. Let $V$ be a dense subspace of  $H$, endowed with a Hilbert structure such that the inclusion map of $V$ into $H$ is compact. Then $H$ is included in $V^{\prime}$ with continuous imbedding. By $\left\Vert \cdot \right\Vert_{V}$ and $( \cdot ,\cdot )_V$ we denote the norm and the scalar product in $V$, respectively. We will denote by
$\langle\cdot,\cdot\rangle$ the duality product between
$V^{\prime} $ and $V$.\\

Let $A\in \mathcal{L}(V,V^{\prime})$ be the duality map: $V\rightarrow V'$. It is a self-adjoint monotone operator such that $A^{-1}\in \mathcal{L}(V^{\prime}, V) \subset \mathcal{L}(H, H)$ is a compact, positive, self-adjoint operator from $H$ to itself. \\

As a consequence of the Hilbert-Schmidt  Theorem  there exists a nondecreasing sequence
of positive real numbers,
\[
0<\lambda_{1}\leq\lambda_{2}\leq...\leq\lambda_j\leq......\text{,}%
\]
with $\lim_{j\to\infty}\lambda_j=+\infty$ and  there exists
an orthonormal basis $\left\{  w_j:j\geq1\right\}$ of $H$ with 
$A\, w_j$ $=$ $\lambda_j \: w_j$ for all $ j$ $\ge$ $1$. The sequence $(\lambda_j)$ is the sequence of eigenvalues repeated according to their multiplicity. \bigskip 

We now assume that $(\lambda_j)$ satisfies the following growth assumption:
\begin{itemize}
\item[(H1)] There exist positive constants $c$ and $\alpha$ such that $$ \lambda_{j} \ge c j^{\alpha}.$$
\end{itemize}

Under the last assumption, the first goal in this section is to prove an estimate of the Kolmogorov $\varepsilon$-entropy of $B_{V}(0,1):=\{ u\in V, \left\Vert u \right\Vert_{V} \leq 1  \}$. In order to do that, we shall identify $H$ with $l^2$ through the identification 
\begin{eqnarray*}
u\longrightarrow \left(u_j \right)_{j\in \mathbb{N}^{*}}
\end{eqnarray*}
where $u= \sum_{j}u_jw_j$.

\begin{theorem}\label{theorem_entropy}
Assume the assumption (H1). Then, the Kolmogorov $\varepsilon$-entropy of $B_{V}(0,1)$ in $H$ satisfies
\begin{eqnarray}\label{main}
\forall \varepsilon>0, \quad \mathcal{H}_{\varepsilon}(B_{V}(0,1),H)< \left(\frac{\log 3+\alpha}{\log 2}\right)\left(\frac{2}{c\varepsilon^2} \right)^{1/\alpha} .
\end{eqnarray}

\end{theorem}
\begin{proof}
Let $u\in B_{V}(0,1)$. We observe that  $$u\in B_{V}(0,1)\Longleftrightarrow \displaystyle \sum_{j=1}^{\infty} \lambda_j u_j^2 \leq1.$$ Let $W\subset H$ be the Hilbert space of vectors $u$ for which $\sum_{j}cj^{\alpha}u_j^2<\infty$ with the norm $\left\Vert u \right\Vert_{W}=\left(c\sum_{j}j^{\alpha}u_j^2 \right)^{1/2}$. Then
 $$u\in B_{W}(0,1)\Longleftrightarrow \displaystyle \sum_{j=1}^{\infty} cj^{\alpha} u_j^2 \leq1.$$ Using (H1), we have that $B_{V}(0,1)\subset B_{W}(0,1)$ and therefore
\begin{eqnarray}\label{relation_entropy_V}
\mathcal{H}_{\varepsilon}(B_{V}(0,1),H)\leq\mathcal{H}_{\varepsilon}(B_{W}(0,1),H).
\end{eqnarray}
If we denote $\mu_j=c^{-1}j^{-\alpha}$, we can write $B_{W}(0,1)$ as an ellipsoid given by $$\mathcal{E}=\{ (u_j)_{j=1}^{\infty}\: | \: \displaystyle\sum_{j=1}^{\infty} \frac{u_j^2}{\mu_j}\leq 1 \}.$$
For a given $\varepsilon>0$, let us give first an upper bound for $N_{\sqrt{2}\varepsilon}(\mathcal{E}, H)$. Let $d$ be the smallest integer such that $\mu_{d+1}\leq \varepsilon^2$. We consider the truncated ellipsoid $$\mathcal{\widetilde E}=\{ u\in \mathcal{E}\:  | \: u_j=0 \quad \text{for all}\quad  j\ge d+1 \}.$$
Given  any $\varepsilon$-cover $\{u^1,...,u^N\}$ of $\mathcal{\widetilde E}$, i.e. for each $u\in \mathcal{\widetilde E}$, there exists some $i\in\{1,...,N\}$ such that 
\begin{equation*}
\sum_{j=1}^d(u_j-u_j^i)^2\leq \varepsilon^2.
\end{equation*}
For any $u\in \mathcal{E}$, we have $$\sum_{j=d+1}^{\infty} u_j^2\leq \mu_{d+1} \sum_{j=d+1}^{\infty}\frac{u_j^2}{\mu_j}\leq \varepsilon^2,$$ and hence for some $i\in\{1,...,N\}$, $$\sum_{j=1}^{\infty}( u_j-u_j^i)^2=\sum_{j=1}^{d}( u_j-u_j^i)^2+\sum_{j=d+1}^{\infty}u_j^2\leq 2\varepsilon^2.$$
Therefore, $\{u^1,...,u^N\}$ forms a $\sqrt{2}\varepsilon$-cover of the full ellipsoid $\mathcal{E}$.  We now view $\mathcal{ \widetilde E}$ as a subset of $\mathbb{R}^d$, i.e. $$\mathcal{ \widetilde E}=\{ u\in \mathbb{R}^d\: | \: \displaystyle\sum_{j=1}^{d} \frac{u_j^2}{\mu_j}\leq 1 \},$$ and we prove the inequality \begin{equation}\label{relation_Volume-Entropy}
N_{\varepsilon}(\mathcal{ \widetilde E},\mathbb{R}^d)\leq \frac{\text{vol}(\mathcal{ \widetilde E}+\frac{\varepsilon}{2}\mathbb{B}_d(1))}{\text{vol}(\frac{\varepsilon}{2}\mathbb{B}_d(1))},
\end{equation}
where $\mathbb{B}_d(1)=\{u\in \mathbb{R}^d \: | \: \sum_{j=1}^d \left\vert u_j\right\vert^2 \leq 1\}$. 

\bigskip 
\noindent 
The proof of \eqref{relation_Volume-Entropy} is actually simple: first of all let us consider any finite family of points $ A = \{(a_i)_{i\in J}\}\subset \mathcal{ \widetilde E}$ for which all balls $ B(a_i, \frac{\varepsilon}{2})$ are pairwise disjoint. Then we have $$\bigcup_i B(a_i, \frac{\varepsilon}{2}) \subset \mathcal{ \widetilde E} + \frac{\varepsilon}{2}\mathbb{B}_d(1), $$ 

hence

$$ \text{vol}(\frac{\varepsilon}{2}\mathbb{B}_d(1)) \text{card} ({A}) =  \sum_{i\in J}\text{vol}(B(a_i, \frac{\varepsilon}{2}))
\le \text{vol}(\mathcal {\widetilde E} + \frac{\varepsilon}{2}\mathbb{B}_d(1)).$$
To conclude, it is sufficient to remark that since the cardinality of such finite sets is bounded, we can consider such a set $A$ with maximal cardinality. Then for any $a\not \in A$ in $\mathcal{\widetilde E}$ , the ball $B(a,  \frac{\varepsilon}{2})$ intersects at least one of the balls $B(a_{k},  \frac{\varepsilon}{2}),$ implying $||a-a_k||\le \varepsilon$. It follows that the balls $ B(a_i, \varepsilon)$ with $a_i\in A$ give an $\varepsilon$-covering of $\mathcal {\widetilde E}$. The result follows immediately.

\bigskip 
\noindent Since $\varepsilon^2< \mu_j$ for all $j\in \{1,...,d\}$, we can see that $\mathcal{ \widetilde E}$ contains the ball $\varepsilon\mathbb{B}_d(1)=\{u\in \mathbb{R}^d \: | \: \sum_{j=1}^d \left\vert u_j\right\vert^2 \leq \varepsilon^2\}$,  hence 
\begin{equation}\label{relation_aux}
\text{vol}(\mathcal{ \widetilde E}+\frac{\varepsilon}{2}\mathbb{B}_d(1))\leq \text{vol}(\frac{3}{2}\mathcal{ \widetilde E}).
\end{equation}
From  (\ref{relation_Volume-Entropy}) and (\ref{relation_aux}), we deduce \begin{equation*}
N_{\varepsilon}(\mathcal{ \widetilde E},H)\leq \frac{3^{d}}{\varepsilon^d}\frac{\text{vol}(\mathcal{ \widetilde E})}{\text{vol}(\mathbb{B}_d(1))}.
\end{equation*}
Since the ellipsoid $\mathcal{ \widetilde E}$ is the image of the  the unit ball  by the linear transform $$ (x_1, ...,x_d ) \longrightarrow (\sqrt{\mu_1}{x_1}, ...,\sqrt{\mu_d}{x_d})$$   it follows classically that  $$\frac{\text{vol}(\mathcal{ \widetilde E})}{\text{vol}(\mathbb{B}_d(1))}=\prod_{j=1}^d\sqrt{\mu_j}$$ and we can deduce that
\begin{equation*}
N_{\varepsilon}(\mathcal{ \widetilde E},H)\leq \frac{3^{d}}{\varepsilon^d}\prod_{j=1}^d(cj^{\alpha})^{-1/2}=\frac{3^{d}}{\varepsilon^d}\left(\frac{1}{\sqrt{c}} \right)^d\left(d! \right)^{-\alpha/2}\leq \frac{3^{d}}{\varepsilon^d}\left(\frac{1}{\sqrt{c}} \right)^d\left(\frac{d}{e} \right)^{-\frac{\alpha}{2}d},
\end{equation*}
using the fact that $\mu_j=c^{-1}j^{-\alpha}$ and the elementary inequality $\displaystyle d! \ge \left(\frac{d}{e} \right)^d$.

Since $\mu_{d+1}\leq \varepsilon^2$, we deduce that 
\begin{equation*}
N_{\varepsilon}(\mathcal{ \widetilde E},H)\leq 3^{d}\left(\frac{d+1}{d} \right)^{\frac{\alpha}{2}d}e^{\frac{\alpha}{2}d}\leq 3^{d}e^{\alpha d},
\end{equation*}
and therefore
\begin{equation*}
\mathcal{H}_{\varepsilon}(\mathcal{ \widetilde E},H)\leq d\left(\frac{\log 3+\alpha}{\log 2}\right).
\end{equation*}
Since $\varepsilon^2< \mu_d=c^{-1}d^{-\alpha}$, we have $d< c^{-1/\alpha}\varepsilon^{-2/\alpha}$, and we obtain
\begin{equation*}
\mathcal{H}_{\varepsilon}(\mathcal{ \widetilde E},H)< \left(\frac{\log 3+\alpha}{\log 2}\right)\left(\frac{1}{c\varepsilon^2} \right)^{1/\alpha}.
\end{equation*}
Then we deduce \begin{equation*}
\mathcal{H}_{\varepsilon\sqrt2}(\mathcal{E},H)< \left(\frac{\log 3+\alpha}{\log 2}\right)\left(\frac{1}{c\varepsilon^2} \right)^{1/\alpha}.
\end{equation*}  So that   by an obvious change of notation\begin{equation*}
\mathcal{H}_{\varepsilon}(\mathcal{E},H)< \left(\frac{\log 3+\alpha}{\log 2}\right)\left(\frac{2}{c\varepsilon^2} \right)^{1/\alpha}.
\end{equation*} and  (\ref{relation_entropy_V}) completes the proof.
\end{proof}

\begin{remark}
This upper bound is rather sharp: for a lower bound of the entropy, we observe that the ellipsoid $\mathcal{E}$ contains the truncated ellipsoid $\mathcal{ \widetilde E}$, which contains the ball $\varepsilon\mathbb{B}_d(1)$. Then, we have
\begin{equation}\label{lower}
N_{\varepsilon \over 2}(\mathcal{E},H)\ge N_{\varepsilon \over 2}(\varepsilon\mathbb{B}_d(1),H)\ge 2^d,
\end{equation}
as a consequence of the obvious inequality 
$$ \text{card} ({A}){\rm vol}({\varepsilon \over 2}\mathbb{B}_d(1))\ge {\rm vol}(\varepsilon\mathbb{B}_d(1))= 2^d {\rm vol}({\varepsilon \over 2}\mathbb{B}_d(1))$$ valid for any ${\varepsilon \over 2}$- covering of $ \varepsilon\mathbb{B}_d(1)$ in $H$ with centers forming the set $A$. Indeed the orthogonal projections in $H$ of the covering balls on the $d$-dimensional space are covering balls of the projection (equal to $ \varepsilon\mathbb{B}_d(1)$) with centers in the $d$-dimensional space and the matter is reduced to $d$ dimensions. 
When $ \lambda_{j} \le C j^{\alpha},$ we can deduce
\begin{eqnarray}\label{relation2_entropy_V}
\mathcal{H}_{\varepsilon}(B_{V}(0,1),H)\ge\mathcal{H}_{\varepsilon}(\mathcal{E},H).
\end{eqnarray}
Since $\varepsilon^2\ge \mu_{d+1}=C^{-1}(d+1)^{-\alpha} $, we have $d\ge C^{-1/\alpha}\varepsilon^{-2/\alpha}-1$, and from (\ref{lower}), we obtain
\begin{equation*}
\mathcal{H}_{\varepsilon \over 2}(\mathcal{E},H)\ge \left(\frac{1}{C\varepsilon^2} \right)^{1/\alpha}-1,
\end{equation*}  
and by an obvious change of notation
\begin{equation*}
\mathcal{H}_{\varepsilon}(\mathcal{E},H)\ge \left(\frac{1}{4C\varepsilon^2} \right)^{1/\alpha}-1.
\end{equation*}  
Therefore, from (\ref{relation2_entropy_V}), we obtain
\begin{equation*}
\mathcal{H}_{\varepsilon}(B_{V}(0,1),H)\ge \left(\frac{1}{4C\varepsilon^2} \right)^{1/\alpha}-1.
\end{equation*}  which is not so far from (\ref{main}).
\end{remark}

Now, consider
\begin{eqnarray}\label{equation_g}
Au+g(u)=\lambda u,
\end{eqnarray}
where $g:V\rightarrow V^{\prime}$ is a continuous nondecreasing function and $\lambda$ is a positive constant.
\\

We define by $\mathcal{C}$ the set of equilibria of (\ref{equation_g}) and we consider the identity map $I:\mathcal{C}\rightarrow \mathcal{C}$.

The second goal in this section is to estimate the fractal dimension of $\mathcal{C}$. First, we prove the following result.
\begin{proposition}\label{proposition_estimate}
For all $u,v\in \mathcal{C}$, 
\begin{eqnarray*}
\left\Vert u-v \right\Vert_{V} \leq \sqrt{\lambda} \left\vert u-v \right\vert_{H}.
\end{eqnarray*}
\end{proposition}
\begin{proof}
Let $u$ and $v$ belong to $\mathcal{C}$ and set $u-v$, where $u$ and $v$ are solutions to (\ref{equation_g}). Then, we obtain
\begin{eqnarray*}
\left\langle A u-A v,u-v \right\rangle+\left\langle g(u)-g(v),u-v \right\rangle=\lambda \left\vert u-v \right\vert^2_{H}.
\end{eqnarray*}
Since  $g$ is non-decreasing, the conclusion follows easily.
\end{proof}
\\

Finally, using Theorems \ref{Zelik_2000} and \ref{theorem_entropy} together with Proposition \ref{proposition_estimate}, we deduce the following result.
\begin{theorem}\label{theorem_dimension}
Assume the assumption (H1). Then, any compact subset  $K \subset \mathcal{C}$ has a finite fractal dimension with
\begin{equation*}
dim_{F} \, K < \left(\frac{\log 3+\alpha}{\log 2}\right)\left(\frac{32 \lambda}{c} \right)^{1/\alpha}.
\end{equation*}
\end{theorem}
\begin{remark}
As we shall see in the next section, in the applications to concrete elliptic equations, the function $ N(\lambda ) = \min\left\{ n\in \mathbb{N}, \lambda_{n+1}\ge  \lambda \right\} $ behaves like some positive power of $\lambda$  for large values of $\lambda$ . The following example now shows that for general monotonic maps $g$, the estimate given by Theorem \ref{theorem_dimension} is optimal up to a multiplicative constant in such a case, therefore essentially optimal as far as the growth as a function of $\lambda$ is concerned and a general monotone map $g$ is allowed. Let us consider $\lambda>0$, $n\in \mathbb{N}$ such that $ \lambda_n< \lambda \le \lambda_{n+1}$ and set 
$$ g(u): = \sum _{j=1}^n (\lambda- \lambda_{j}) P_j(u),$$ where $P_j$ is the orthogonal projection from $H$ to the eigenspace of $A$ corresponding to the eigenvalue $\lambda_{j}$. Now the equation
 $$ Au + g(u)= \lambda u $$ reduces to 
 $$ Au = \sum _{j=1}^n  \lambda_{j} P_j(u) +\lambda( u - \sum _{j=1}^n  P_j(u)) , $$ so that 
 $$X_n = \bigoplus _{j = 1}^nP_j(H)  = \bigoplus _{j=1}^n \ker(A- \lambda_{j}I)\subset  \mathcal{C}.$$
 Consequently, in this case $\mathcal{C}$ contains a vector space of dimension 
 $$ d= card \left\{ j\in \mathbb{N}^*, \lambda> \lambda_j \right\} = N(\lambda).$$ In particular for the unit ball $K$ of this finite dimensional space, which is a compact subset of $\mathcal{C}$ we find $$ dim_{F} K\, \ge N(\lambda).
$$ 
When $ \lambda_{k} \le C k^{\alpha}, $ then $N(\lambda) \ge \left(\frac{\lambda}{C} \right)^{1/\alpha} -1 .$ This confirms the optimality of the upper estimate up to a constant for $\lambda$ large.

\end{remark}

\section{Application to some elliptic equations}
Let $\Omega \subset \mathbb{R}^N$, $N\ge 1$, be a bounded open domain with  sufficiently smooth boundary. 
We denote by $ \left(  \cdot,\cdot\right)$ the inner product in $L^2(\Omega)$, and by $\left| \cdot\right|_{L^{2}\left(\Omega\right)}$ the associated norm. By $\left|\left|
\cdot \right|\right|_{H^1_0(\Omega)}$ we denote the norm in $H^1_0(\Omega)$, which is associated to the inner product $\left( \left( \cdot,\cdot \right)\right):=\left(\nabla\cdot , \nabla \cdot\right).$ We will denote by
$\langle\cdot,\cdot\rangle$ the duality product between
$H^{-1}\left(  \Omega\right) $ and $H_{0}^{1}\left(  \Omega\right)
$. By $\left\Vert \cdot \right\Vert_{L^{\infty}(\Omega)}$ we denote the norm in $L^{\infty}(\Omega)$.\\

Let $A=-\Delta_D$ with $\Delta_D$ the Dirichlet Laplacian on $\Omega$. We denote by $\lambda_1$ the first eigenvalue of the $A$. Let $\lambda_j$ denote the j$^{\text{th}}$ eigenvalue of $\Omega$ for the Dirichlet boundary problem. We use the estimate (see Li and Yau \cite{Li_Yau} for more details)
\begin{equation}\label{Li_Yau_formula}
\lambda_j\ge \frac{NC_N}{N+2}j^{2/N}V^{-2/N},
\end{equation}
where $V$ is the volume of $\Omega$, and $C_N=(2\pi)^2B_N^{-2/N}$, with $B_N=$ volume of the unit $N$-ball.

Taking into account (\ref{Li_Yau_formula}), in particular, we have (H1) with $\alpha=\frac{2}{N}$ and $c=\displaystyle \frac{4\pi N}{N+2} (\Gamma(1+\frac{N}{2}))^{2/N}\left\vert \Omega \right\vert^{-2/N}$, where $\left\vert \Omega \right\vert$ denotes the $N$-dimensional measure of $\Omega$. As a consequence we find the following result.
\begin{corollary}  Let $g$ be any nondecreasing continuous function of the real variable $s$ with super-linear growth at infinity. Let  $\mathcal{C}$ be the set of solutions of the equation 

$$ u \in H^1_0(\Omega)\cup L^{\infty}(\Omega); \quad -\Delta u + g(u) = \lambda u .$$

Then $\mathcal{C}$ is  compact with a finite fractal dimension such that  
\begin{equation*} dim_{F} \, \mathcal{C} < 32^{N/2}\left(\frac{\log 3+\frac{2}{N}}{\log 2}\right)\left(\frac{N+2}{4\pi N} \right)^{N/2} \frac{1}{\Gamma\left(1+N/2 \right)}\left\vert \Omega \right\vert\lambda^{N/2}.
\end{equation*}
\end{corollary}\begin{proof} Compactness is an immediate consequence of super-linear growth at infinity. Then it is sufficient to apply Theorem \ref{theorem_dimension} with  $K = \mathcal{C}$
\end{proof}

\section{Application to parabolic equations}
Now, we consider the following problem
\begin{eqnarray}\label{example2_general}
u_t-\Delta u+g(u)=\lambda u \text{\ \ in \ } \Omega,
\end{eqnarray}
with the zero Dirichlet boundary condition,
\begin{eqnarray}\label{example2_boundary_general}
u=0 \text{\ \ on \ } \partial \Omega,
\end{eqnarray}
and the initial condition
\begin{eqnarray}\label{example2_initial_general}
u(x,0)=u_0(x), \text{\ for \ } x\in \Omega,
\end{eqnarray}
where $\lambda$ is a positive constant and $g\in C^1(\mathbb{R})$ is a non-decreasing function. We assume that the non-linear term $g$ satisfies a dissipativity assumption of the form
\begin{eqnarray}\label{condition2_g}
g(s)s\ge\beta \left\vert s \right\vert^p \text{\ \ \ \ } \forall s\in\mathbb{R},
\end{eqnarray}
and the following growth restriction of the derivative
\begin{eqnarray}\label{condition3_g}
\left\vert g(s_1)-g(s_2)\right\vert \leq \gamma\, M^{p-2} \left\vert s_1-s_2 \right\vert  \text{\ \ for \ \ }  |s_1|,|s_2|\leq M,
\end{eqnarray}
for some $\beta>0$, $\gamma>0$, $M>0$ and $p> 2$. A typical example of a function satisfying the previous conditions is $g(s)=\beta \left\vert s \right\vert^{p-2}s$, with $p> 2.$ In this case we may take $\gamma = \beta(p-1)$.
\\

We define a semigroup $\{ S(t),t\ge 0 \}$ in $L^2(\Omega)$ by\begin{eqnarray}\label{def_semigroup}
S(t)u_0=u(t;0,u_0) \text{\ \ \ } \forall u_0 \in L^2(\Omega), \text{\ \ \ } \forall t\ge0,
\end{eqnarray}
where $u(t;0,u_0)$ is the unique solution of (\ref{example2_general})-(\ref{example2_initial_general}). We denote by $\mathcal{A}$ the global attractor associated with the semigroup $S$ defined by (\ref{def_semigroup}).
\\

Our aim is to estimate the fractal dimension of $\mathcal{A}$. First, we need the following results.
\begin{proposition}\label{estimSolutionL2_general}
Assume (\ref{condition2_g}). Then the attractor $\mathcal{A}$ associated with (\ref{example2_general})-(\ref{example2_initial_general}) is bounded in $L^{2}(\Omega)$. More concretely, there exists a positive constant $C(p,\beta,\lambda,\Omega)$ such that
\begin{eqnarray*}
\left\vert a \right\vert_{L^{2}(\Omega)}\leq C, \text{\ \ \ \ } \forall a\in \mathcal{A}.
\end{eqnarray*} 
\end{proposition}
\begin{proof}
Multiplying (\ref{example2_general}) by $u$,
\begin{eqnarray*}
\frac{1}{2}\frac{d}{dt}\left(\left\vert u \right\vert^2_{L^2(\Omega)} \right)+\left\vert \nabla u\right\vert^2_{L^2(\Omega)}=\left(\lambda u-g(u),u \right).
\end{eqnarray*}
Using (\ref{condition2_g}) and Young's inequality applied with the conjugate exponents $\frac{p}{2}$ and $\frac{p}{p-2}$, we have
\begin{eqnarray*}
\left(\lambda u-g(u),u \right)\leq \displaystyle \int_{\Omega}\left(\lambda u^2-\beta \left\vert u \right\vert ^p \right)dx\leq \frac{2(p-2)}{p^2}\frac{1}{\beta}\left\vert \Omega \right\vert \lambda^{\frac{p}{p-2}},
\end{eqnarray*}
then, using the Poincar\' e inequality, we obtain
\begin{eqnarray*}
\frac{d}{dt}\left(\left\vert u \right\vert^2_{L^2(\Omega)} \right)+2\lambda_1\left\vert u\right\vert^2_{L^2(\Omega)}\leq \frac{4(p-2)}{p^2}\frac{1}{\beta}\left\vert \Omega \right\vert \lambda^{\frac{p}{p-2}}.
\end{eqnarray*}
Multiplying by $e^{2\lambda_1 t}$ and integrating between $0$ and $t$, we obtain
\begin{eqnarray*}
\left\vert u(t) \right\vert^2_{L^2(\Omega)} \leq e^{-2\lambda_1 t}\left\vert u_0 \right\vert^2_{L^2(\Omega)}+\frac{2(p-2)}{p^2 }\frac{1}{\beta \lambda_1}\left\vert \Omega \right\vert \lambda^{\frac{p}{p-2}}\left(1-e^{-2\lambda_1 t} \right).
\end{eqnarray*}
We observe that if $u(t)\in \mathcal{A}$, then there exists $u_0\in \mathcal{A}$ such that $u(t)=S(t)u_0$. Then, we have
\begin{eqnarray*}
\left\vert S(t)u_0 \right\vert^2_{L^2(\Omega)} \leq e^{-2\lambda_1 t}\left\vert u_0 \right\vert^2_{L^2(\Omega)}+\frac{2(p-2)}{p^2}\frac{1}{\beta \lambda_1}\left\vert \Omega \right\vert \lambda^{\frac{p}{p-2}}\left(1-e^{-2\lambda_1 t} \right).
\end{eqnarray*}
Fix $t>0$, and consider $a\in \mathcal{A}$. Then, there exists $u_0\in \mathcal{A}$ such that $a=S(t)u_0$, and we have
\begin{eqnarray*}
\left\vert a \right\vert^2_{L^2(\Omega)} \leq e^{-2\lambda_1 t}\left\vert u_0 \right\vert^2_{L^2(\Omega)}+\frac{2(p-2)}{p^2}\frac{1}{\beta \lambda_1}\left\vert \Omega \right\vert \lambda^{\frac{p}{p-2}}\left(1-e^{-2\lambda_1 t} \right).
\end{eqnarray*}
If $t$ tends to $+\infty$, we obtain
\begin{eqnarray*}
\left\vert a \right\vert^2_{L^2(\Omega)} \leq \frac{2(p-2)}{p^2}\frac{1}{\beta \lambda_1}\left\vert \Omega \right\vert \lambda^{\frac{p}{p-2}} \text{\ \ \ \ } \forall a\in\mathcal{A}.
\end{eqnarray*}\end{proof}

Taking into account Proposition \ref{estimSolutionL2_general}, we prove the following result
\begin{proposition}\label{estimSolution_general}
Assume (\ref{condition2_g}). Then the attractor $\mathcal{A}$ associated with (\ref{example2_general})-(\ref{example2_initial_general}) is uniformly bounded in $L^{\infty}(\Omega)$. More precisely,
\begin{eqnarray}\label{estimate_attractor_g}
\left\Vert a \right\Vert_{L^{\infty}(\Omega)}\leq \left(\frac{\lambda}{\beta}\right)^{\frac{1}{p-2}}, \text{\ \ \ \ } \forall a\in \mathcal{A}.
\end{eqnarray} 
\end{proposition}
\begin{proof}
Using (\ref{condition2_g}), we observe
\begin{eqnarray}\label{condition_positive}
g(s)-\lambda s\ge 0 \text{\ \ \ \ if \ \ \ }s\ge M:=\left(\frac{\lambda}{\beta} \right)^{\frac{1}{p-2}}.
\end{eqnarray}
Let $u\in L^2(\Omega)$, we define
\begin{equation*}
u_{+}(x)=\left\{
\begin{array}
[c]{l}%
u(x)\text{, \ if }u(x)>0\text{,}\\
0\text{, in other case,}%
\end{array}
\right.  \label{u+}%
\end{equation*}%
and
\begin{equation*}
u_{-}(x)=\left\{
\begin{array}
[c]{l}%
u(x)\text{, \ if }u(x)<0\text{,}\\
0\text{, in other case.}%
\end{array}
\right.  \label{u-}%
\end{equation*}
Multiplying (\ref{example2_general}) by $\left(u(x)-M \right)_{+}$, taking into account (\ref{condition_positive}) and using the Poincar\' e inequality, we have
\begin{eqnarray*}
\frac{d}{dt}\left(\left\vert \left(u-M \right)_{+}\right\vert^2_{L^2(\Omega)} \right) \leq -2\lambda_1\left\vert \left(u-M \right)_{+}\right\vert^2_{L^2(\Omega)}.
\end{eqnarray*}
Multiplying by $e^{2\lambda_1 t}$ and integrating between $0$ and $t$, we obtain
\begin{eqnarray*}
\left\vert \left(u(t)-M \right)_{+}\right\vert^2_{L^2(\Omega)}  \leq e^{-2\lambda_1t}\left\vert \left(u_0-M \right)_{+}\right\vert^2_{L^2(\Omega)}.
\end{eqnarray*}
We observe that if $u(t)\in \mathcal{A}$, then there exists $u_0\in \mathcal{A}$ such that $u(t)=S(t)u_0$. Then, we have
\begin{eqnarray*}
\displaystyle \int_{\Omega}(S(t)u_0-M)_+^2dx\leq e^{-2\lambda_1t} \int_{\Omega}(u_0-M)_+^2dx.
\end{eqnarray*}
As $\mathcal{A}$ is bounded in $L^2(\Omega)$, then we can deduce that there exists a positive constant $\widehat{C}(p,\beta,\lambda,\Omega)$, which is independent of $u_0$, such that
\begin{eqnarray*}
\displaystyle \int_{\Omega}(u_0-M)_+^2dx\leq \int_{\Omega}(u_0-M)^2dx\leq 2\left(\left\vert u_0 \right\vert^2_{L^2(\Omega)}+M^2\left\vert \Omega \right\vert \right)\leq \widehat{C},
\end{eqnarray*}
and, we have
\begin{eqnarray*}
\displaystyle \int_{\Omega}(S(t)u_0-M)_+^2dx\leq \widehat{C}e^{-2\lambda_1 t}.
\end{eqnarray*}
Fix $t>0$, and consider $a\in \mathcal{A}$. Then, there exists $u_0\in \mathcal{A}$ such that $a=S(t)u_0$, and we have
\begin{eqnarray*}
0\leq\displaystyle \int_{\Omega}(a-M)_+^2dx\leq \widehat{C}e^{-2\lambda_1 t}.
\end{eqnarray*}
If $t$ tends to $+\infty$, we obtain
\begin{eqnarray*}
\displaystyle \int_{\Omega}(a-M)_+^2dx=0 \text{\ \ \ \ } \forall a\in\mathcal{A},
\end{eqnarray*}
then
\begin{eqnarray*}
(a-M)_+=0 \Longrightarrow a\leq M \text{\ \ \ } \forall a\in \mathcal{A}.
\end{eqnarray*}
We use a similar reasoning for $\left(u+M\right)_{-}$, and then we have 
\begin{eqnarray*}
\left\Vert a \right\Vert_{L^{\infty}(\Omega)}\leq M \text{\ \ \ \ } \forall a\in \mathcal{A}.
\end{eqnarray*}
\end{proof}

Now, we define 
\begin{eqnarray}\label{def_C_gamma}
 C_{\gamma,\beta}:=\displaystyle\frac{1}{2+\frac{\gamma}{\beta}},
\end{eqnarray}
and we consider the map $\displaystyle S\left(\frac{C_{\gamma,\beta}}{\lambda} \right):\mathcal{A}\rightarrow \mathcal{A}$. Taking into account Proposition \ref{estimSolution_general}, we prove the following result.
\begin{proposition}\label{proposition_estimate_parabolic}
Assume (\ref{condition2_g}) and (\ref{condition3_g}). Then, for all $u_0,v_0\in \mathcal{A}$,
\begin{eqnarray*}
  \left\Vert  S\left(\frac{C_{\gamma,\beta}}{\lambda }\right)u_0-S\left(\frac{C_{\gamma,\beta}}{\lambda }\right)v_0 \right\Vert_{H_0^1(\Omega)}\leq   \sqrt{2\lambda\left(1+{\gamma \over \beta} \right)} \left\vert u_0-v_0\right\vert_{L^2(\Omega)},
\end{eqnarray*}
where $C_{\gamma,\beta}$ is given by (\ref{def_C_gamma}).
\end{proposition}
\begin{proof}
Let $u$ and $v$ belong to $\mathcal{A}$ and set $w=u-v$ and $w_0=u_0-v_0$, where $u$ and $v$ are solutions to (\ref{example2_general})-(\ref{example2_boundary_general}) with initial data $u_0$ and $v_0$, respectively. Then, we obtain
\begin{eqnarray}\label{example2_1_g}
w_t-\Delta w+g(u)-g(v)=\lambda w \text{\ \ in \ } \Omega,
\end{eqnarray}
\begin{eqnarray*}
w=0 \text{\ \ on \ } \partial \Omega,
\end{eqnarray*}
\begin{eqnarray*}
w(x,0)=u_0(x)-v_0(x), \text{\ for \ } x\in \Omega.
\end{eqnarray*}
We denote by $W(t):=\left\vert w(t) \right\vert^2_{L^2(\Omega)}$ and by $V(t):=\left\vert \nabla w(t) \right\vert^2_{L^2(\Omega)}$. Multiplying (formally) (\ref{example2_1_g}) by $w$ and taking into account that 
\begin{eqnarray*}
 \displaystyle\int_{\Omega} (g(u)-g(v))wdx\ge0,
\end{eqnarray*}
we obtain
\begin{eqnarray*}\label{deriv}
\frac{1}{2}W'(t)+V(t)\leq \lambda W(t).
\end{eqnarray*} 
Multiplying by $2 e^{-2\lambda t}$, we obtain
\begin{eqnarray*}\label{estim0_g}
\left(e^{-2\lambda t}  W(t)\right)'+2e^{-2\lambda t}V(t) \leq 0.
\end{eqnarray*} 
Integrating between $0$ and $t$, we obtain
\begin{eqnarray*}
e^{-2\lambda t}W(t)+2\int_{0}^{t}e^{-2\lambda s}V(s)ds \leq W(0),
\end{eqnarray*}
yielding
\begin{eqnarray}\label{estim1_g}
W(t)+2\int_{0}^{t}V(s)ds \leq e^{2\lambda t}W(0).
\end{eqnarray}
Now, multiplying (formally) (\ref{example2_1_g}) by $\frac{\partial w}{\partial t}$, we obtain
\begin{eqnarray*}
\left\vert \frac{\partial w}{\partial t} \right\vert ^2_{L^2(\Omega)}+\frac{1}{2}V'(t)+\int_{\Omega} (g(u)-g(v)) \frac{\partial w}{\partial t} dx
=\lambda (w, \frac{\partial w}{\partial t}).
\end{eqnarray*}
We note that, owing to (\ref{condition3_g}), (\ref{estimate_attractor_g}) and H\"older's inequality,
\begin{eqnarray*}
\left\vert \displaystyle\int_{\Omega} (g(u)-g(v))\frac{\partial w}{\partial t}dx\right\vert  \leq  {\gamma \over \beta} \lambda \int_{\Omega} \left\vert w \right\vert \left\vert \frac{\partial w}{\partial t}\right\vert dx
 \leq   \frac{\gamma}{\beta} \lambda W(t)^{1/2}\left\vert \frac{\partial w}{\partial t}\right\vert_{L^2(\Omega)},
 \end{eqnarray*}
and by H\"older's inequality
 \begin{eqnarray*}
\lambda (w,\frac{\partial w}{\partial t})\leq \lambda  W(t)^{1/2}\left\vert \frac{\partial w}{\partial t}\right\vert_{L^2(\Omega)}.
\end{eqnarray*}
Then, by Young's  inequality  we have
\begin{eqnarray*}
\frac{1}{2}V'(t)\leq  {1\over 4}\lambda^2 \left(1+{\gamma \over \beta} \right)^2 W(t).
\end{eqnarray*}

Then, using (\ref{estim1_g}), we obtain
\begin{eqnarray}\label{estimDeriva}
V'(t) \le e^{2\lambda t}{1\over 2}\lambda^2 \left(1+{\gamma \over \beta} \right)^2W(0).
\end{eqnarray} 
Using the equality
$$ V(t) =  \frac{1}{t}\int _0^t V(s) ds  + \frac{1}{t}\int _0^t s\,V' (s) ds,$$ and from (\ref{estim1_g}) and (\ref{estimDeriva}), we deduce 
$$ V(t) \le \lambda W(0)\left[\frac{e^{2\lambda t}}{2\lambda t}+{1\over 4}\left(1+{\gamma \over \beta} \right)^2{2\lambda t\,e^{2\lambda t}-(e^{2\lambda t}-1) \over 2\lambda t}\right] .$$

Taking $t=\displaystyle \frac{C_{\gamma,\beta}}{\lambda}$, where $C_{\gamma, \beta}$ is given by (\ref{def_C_gamma}), we finally deduce from the above inequality an inequality of the form $$ V\left(\frac{C_{\gamma,\beta}}{\lambda}\right)  \le \widetilde{C}_{\gamma,\beta}\, \lambda W(0),$$ with 

\begin{equation*}
\widetilde{C}_{\gamma,\beta}= \frac{e^{2C_{\gamma,\beta}}}{2C_{\gamma,\beta}}  + {C_{\gamma,\beta}\over 2}\left(1+\frac{\gamma}{\beta} \right)^2,
\end{equation*}
where we have used that $\displaystyle {(2C_{\gamma,\beta}-1)e^{2C_{\gamma,\beta}}+1 \over 2C_{\gamma,\beta}}\leq 2C_{\gamma,\beta}$ for $C_{\gamma,\beta}\leq{1\over 2}$.

Finally, we estimate $\widetilde{C}_{\gamma,\beta}$. Taking into account that $\displaystyle {2 \over 3}\left(1+{\gamma \over \beta}\right)^{-1}\leq C_{\gamma,\beta}\leq {1\over 3}$ for the first term and using that $C_{\gamma,\beta}\leq {1 \over 1+{\gamma \over \beta} }$ for the second one, we can deduce
\begin{equation*}
\widetilde{C}_{\gamma,\beta} \leq \left({3\over 4}e^{2\over 3}+{1 \over 2} \right)\left(1+\frac{\gamma}{\beta} \right).
\end{equation*}
\end{proof}
\\

Finally, using Theorems \ref{Zelik_2000} and \ref{theorem_entropy} together with Proposition \ref{proposition_estimate_parabolic} and (\ref {Li_Yau_formula}), we deduce the following result.
\begin{proposition}\label{theorem_dimension_parabolic}
Assume (\ref{condition2_g})-(\ref{condition3_g}). Then, the global attractor $\mathcal{A}$ associated with (\ref{example2_general})-(\ref{example2_initial_general}) has finite fractal dimension in $L^2(\Omega)$, and satisfies 
\begin{eqnarray*}
dim_{F} \mathcal{A} < 8^{N}\left(\frac{\log 3+\frac{2}{N}}{\log 2}\right)\left(\frac{N+2}{4\pi N} \right)^{N/2} \frac{1}{\Gamma\left(1+N/2 \right)}\left\vert \Omega \right\vert\left(1+\frac{\gamma}{\beta} \right)^{N/2}\lambda^{N/2}.
\end{eqnarray*}
\end{proposition}
\begin{remark}
This result is substantially weaker than the estimate obtained in Theorem 3.1. in \cite{Chepy_Ilyin}, but to obtain it we do not need any regularity hypothesis on $g$ stronger than $C^1$ .
\end{remark}
\begin{remark}
We presently do not know if (\ref{condition3_g}) is really needed for our method to be employed. In particular the factor  $ \left(1+\frac{\gamma}{\beta} \right)^{N/2}$ does not appear in the estimate of  \cite{Chepy_Ilyin} and the result of Theorem \ref{theorem_dimension} even suggests that local compactness of the attractor might be a sufficient condition for its fractal dimension  to be finite. This aspect seems to have been overlooked systematically in the literature until now and might be an interesting track of research for the future.
\end{remark}

\subsection*{Acknowledgments}
Mar\'ia Anguiano has been supported by Junta de Andaluc\'ia (Spain), Proyecto de Excelencia P12-FQM-2466, and in part by  European Commission, Excellent Science-European Research Council (ERC) H2020-EU.1.1.-639227.

\end{document}